\documentclass[11pt]{article}

\usepackage[margin=1.05in]{geometry}
\usepackage{amsmath,amssymb,amsthm,mathtools}
\usepackage[hidelinks]{hyperref}
\usepackage[nameinlink,capitalise]{cleveref}
\usepackage{titlesec}

\titleformat{\section}
  {\centering\normalfont\scshape}
  {\thesection.}
  {0.5em}
  {}
  
\newtheorem{theorem}{Theorem}[section]
\newtheorem{proposition}[theorem]{Proposition}
\newtheorem{lemma}[theorem]{Lemma}
\newtheorem{corollary}[theorem]{Corollary}
\theoremstyle{definition}
\newtheorem{definition}[theorem]{Definition}
\theoremstyle{remark}
\newtheorem{remark}[theorem]{Remark}

\newcommand{\R}{\mathbb R}
\newcommand{\cB}{\mathcal B}
\newcommand{\cE}{\mathcal E}
\newcommand{\cH}{\mathcal H}

\newcommand{\cK}{\mathcal K}
\newcommand{\cQ}{\mathcal Q}
\newcommand{\Ind}{\operatorname{Ind}}
\newcommand{\Span}{\operatorname{span}}
\newcommand{\Id}{\operatorname{Id}}

\title{Catenoid-sharp index–topology estimates for minimal hypersurfaces}
\author{Junzhang Li%
\thanks{Chern Institute of Mathematics, Nankai University, Tianjin 300071, China, \href{2112355@mail.nankai.edu.cn}{\texttt{2112355@mail.nankai.edu.cn}}}
}
\date{}
\begin{document}
\maketitle

\begin{abstract}
 We show that for an embedded two-sided
minimal hypersurface in $\R^N$, there is a lower bound for the index in
terms of the first Betti number and the number of ends, using ideas from the recent work of Chodosh-Gianocca.  This estimate is
sharp for the higher-dimensional catenoid. We also obtain a
$\operatorname{Spin}(7)$ analogue of \cite[Theorem~10.1]{CG}: a complete
two-sided minimal immersion $M^7\to\R^8$ of index one and finite total
curvature is a higher-dimensional catenoid.
\end{abstract}

\section{Introduction}
Throughout, let $N=n+1\geq 4$. Consider a complete embedded minimal hypersurface \(M^n\subset\mathbb{R}^{N}\). For \(u\in C_c^\infty(M)\), let
\begin{equation}\label{eq:second-variation}
\mathcal{Q}(u,u)=\int_M |\nabla u|^2-|A|^2u^2
\end{equation}
be the second variation of area, where \(A\) is the second fundamental form, and define the \emph{Morse index} by
\begin{equation}\label{eq:morse-index}
\operatorname{Ind}(M):=\sup\bigl\{\dim U:U\subset C_c^\infty(M),\ \mathcal{Q}(u,u)<0\text{ for all }u\in U\setminus\{0\}\bigr\}.
\end{equation}
We say that \(M\) has \emph{finite total curvature} if \(\int_M |A|^n\,d\mu<\infty\). We write \(b_1(M)\) for its first Betti number and \(k\) for its number of ends.

The relation between Morse index and topology is particularly well understood for minimal surfaces in \(\mathbb{R}^3\). For an upper bound of the index, in dimension two, an earlier quantitative
estimate was obtained by Tysk \cite{Tys87}, while sharper estimates,
valid for complete oriented minimal surfaces of finite total curvature
in Euclidean spaces of arbitrary dimension, were subsequently proved
by Ejiri--Micallef \cite{EM08}. As for the lower bound, Chodosh--Maximo \cite{CM16} first proved that a complete two-sided immersed minimal surface of genus \(g\) with \(k\) ends satisfies
\[
\operatorname{Ind}(M)\geq \frac{2}{3}(g+k)-1,
\]
and used this estimate to rule out complete embedded minimal surfaces of index two. In their sequel \cite{CM23} they proved, when the ends are embedded, the sharper estimate
\begin{equation}\label{eq:CM-bound}
\operatorname{Ind}(M)\geq \frac{1}{3}(2g+4k-5).
\end{equation}
Their fully general estimate also records the multiplicities of nonembedded ends. Among its additional low-index consequences are the nonexistence of a complete two-sided minimal immersion in \(\mathbb{R}^3\) of index two and of a complete embedded minimal surface in \(\mathbb{R}^3\) of index three.

The use of harmonic $1$-forms to construct test functions goes back
to Ros \cite{Ros06} and Savo \cite{Sav10}. Chao Li \cite{Li}
adapted this method to complete higher-dimensional minimal
hypersurfaces. In our notation, his general estimate gives
\begin{equation}\label{eq:Li-index-nullity}
\operatorname{Ind}(M)+\operatorname{nullity}(M)\geq \frac{b_1(M)+k-1}{\binom{N}{2}},
\end{equation}
while his index-only estimate is
\begin{equation}\label{eq:Li-index}
\operatorname{Ind}(M)\geq \frac{b_1(M)+k}{\binom{N}{2}}-\frac{4}{N}.
\end{equation}
The latter holds unconditionally when \(N=4\), and for \(N\geq5\) under the assumption that \(M\) has a point at which all principal curvatures are distinct.

Recently, Chodosh--Gianocca introduced a larger space of harmonic \(1\)-forms together with admissible pairings of ambient \(2\)-forms. Their main theorem states that every complete, connected, embedded minimal hypersurface of finite total curvature and index one is a higher-dimensional catenoid \cite[Theorem~1.1]{CG}. Two results from their final section are especially relevant here. First, in \(\mathbb{R}^4\) the self-dual projection allows embeddedness to be removed: every complete, connected, two-sided minimal immersion \(M^3\to\mathbb{R}^4\) of index one is a catenoid \cite[Theorem~10.1]{CG}. Second, for a non-flat, complete, connected, embedded minimal hypersurface \(M^3\subset\mathbb{R}^4\) of finite index, they prove
\begin{equation}\label{eq:CG-index-bound}
\operatorname{Ind}(M)\geq \frac{b_1(M)+k+1}{3}
\end{equation}
\cite[Theorem~10.2]{CG}. The denominator \(3\), rather than \(\dim\Lambda^2\mathbb{R}^4=6\), comes from the fixed rank-three self-dual pairing.

Motivated by this framework, our first result gives an index--topology estimate in every ambient dimension. It may be viewed as an arbitrary-dimensional version of the finite-index dimension count behind \cite[Theorem~10.2]{CG}, using an adapted admissible pairing in the last negative eigendirection.
\begin{theorem}\label{thm:main}
For \(N\geq 4\), suppose that \(M^n\subset\mathbb{R}^N\) is a non-flat, complete, connected, embedded minimal hypersurface with finite total curvature. Then
\begin{equation}\label{eq:main}
\Ind(M)\geq 1+\frac{b_1(M)+k-2}{\binom{N}{2}}.
\end{equation}
\end{theorem}

For the higher-dimensional catenoid, \(b_1=0\), \(k=2\), and \(\Ind=1\) \cite{TZ09}, so equality holds in \eqref{eq:main}. Thus \eqref{eq:main} is "sharp" on the higher-dimensional catenoid in
every ambient dimension. Compared with C. Li's index estimate \eqref{eq:Li-index}, \eqref{eq:main} has the same coefficient of \(b_1+k\), improves the additive term, and for \(N\geq5\) does not require the existence of a point with distinct principal curvatures. On the other hand, when \(N=4\), the exceptional rank-three pairing makes \eqref{eq:CG-index-bound} stronger than the specialization of \eqref{eq:main}.

The denominator \(\binom{N}{2}\) in \eqref{eq:main} is not expected to be optimal. In ambient dimension eight, a fixed Cayley \(4\)-form \(\Phi\) gives the \(\operatorname{Spin}(7)\)-decomposition
\[
\Lambda^2\mathbb{R}^8=\Lambda^2_7\oplus\Lambda^2_{21}
\]
and hence a rank-seven admissible pairing. In the embedded finite-index setting, this reduces the number of test functions from \(\dim\Lambda^2\mathbb{R}^8=28\) to \(7\) and yields the following stronger estimate.
\begin{theorem}\label{thm:spin7-index-bound}
Suppose that \(M^7\subset\mathbb{R}^8\) is a non-flat, complete, connected, embedded minimal hypersurface with finite total curvature. Then
\begin{equation}\label{eq:spin7-index-bound}
\Ind(M)\geq\frac{b_1(M)+k+5}{7}.
\end{equation}
\end{theorem}
The same fixed pairing is compatible with the translation-field argument of \cite[Theorem~10.1]{CG} and therefore also treats immersed hypersurfaces whose ends need not be parallel.
\begin{theorem}\label{thm:spin7-index-one}
Suppose that \(M^7\to\mathbb{R}^8\) is a complete, connected, two-sided minimal immersion with finite total curvature and \(\Ind(M)=1\). Then \(M\) is a higher-dimensional catenoid.
\end{theorem}

In addition to the result of \cite{CG} in $\mathbb R^4$,
there are classical index-one classification results in $\mathbb R^3$.
 Cheng--Tysk
\cite{CT88} proved the catenoid characterization in $\mathbb R^3$
under the assumption that the ends are embedded, while López--Ros
\cite{LR89} showed that the catenoid and Enneper's surface are the only
complete two-sided minimal immersions in $\mathbb R^3$ of index one.

Compared to \cite{CG}, here we need to assume finite total curvature since the relation between finite index and finite total curvature depends sharply on the ambient dimension. Finite total curvature implies finite index in every dimension \cite{FC85,Tys89}. Conversely, finite index implies finite total curvature for complete two-sided minimal immersions when \(3\leq N\leq6\): \cite{FC85,Gul86,CL24,CL23,CMR24,CLMS24,Maz24}. However, this fails in every ambient dimension \(N\geq8\). Simons exhibited the relevant stable cone in \(\mathbb{R}^8\) \cite{Sim67}, Bombieri--De Giorgi--Giusti proved that the Simons cone is area-minimizing \cite{BDGG69}, and the Hardt--Simon theory produces smooth complete area-minimizing hypersurfaces asymptotic to such cones \cite{HS85}. Taking Euclidean products in the remaining dimensions, one obtains, for every \(N\geq8\), a complete two-sided minimal hypersurface \(M^{N-1}\subset\mathbb{R}^N\) with \(\Ind(M)=0\) but \(\int_M|A|^{N-1}=\infty\). Thus the finite-total-curvature assumptions in our two ambient-eight results are genuinely necessary. The remaining case \(N=7\) is open in general.

Finally, we briefly describe the proof strategy. \cref{thm:main} combines the proof strategy of both \cite{CG} and \cite{Li}. Compared to \cite[Theorem1.1]{CG} we introduce more than one bound state in order to obtain an index estimate. More precisely, we first follow \cite{Li} to investigate into all bound states but one by a rough estimate (by taking the admissible pairings to be the whole orthonormal basis of $\Lambda^2\mathbb{R}^N$) and then mimic the proof of \cite[theorem 1.1]{CG} to handle the remaining one. For the \(\operatorname{Spin}(7)\) results instead use a fixed rank-seven pairing, in direct analogy with the self-dual rank-three pairing and translation-field argument in \cite[Section~10]{CG}.

\paragraph{Organization.}
In \cref{sec:finite-index-decomposition} we establish the finite-index spectral decomposition used throughout the paper. In \cref{sec:CG-inputs} we collect, in the notation needed here, the specific results from Sections~4--8 of Chodosh--Gianocca~\cite{CG} and prove their finite-index orthogonality consequence. We prove \cref{thm:main} in \cref{sec:main-proof}. Finally, \cref{sec:spin7} describes the \(\operatorname{Spin}(7)\)-decomposition of two-forms, proves the improved index estimate \cref{thm:spin7-index-bound}, and proves the immersed index-one result \cref{thm:spin7-index-one}.

\paragraph{Use of AI.}
The ideas that \cite{CG}'s method could be applied to get such an index estimate and to use the \(\operatorname{Spin}(7)\)-decomposition were suggested to the author by a large language model. The author directed this use, selected and substantially revised any generated text, and independently verified the mathematical statements and proofs. The author takes full responsibility for the final content of the paper.

\paragraph{Acknowledgements.}
The author is very grateful to Qiongling Li for bringing this question
to the author's attention and for suggesting and greatly assisting with
revisions to an earlier draft.  The author would also like to thank
Otis Chodosh for his interest in this work and for helpful comments and
suggestions, and Yiyang Xiao and Zunpeng Zhou for helpful discussions.

\section{The finite-index decomposition}\label{sec:finite-index-decomposition}

We follow exactly the sign conventions of~\cite[Sections~2.0.3 and 3.2]{CG}:
\[
 I:=Ind(M),
 \qquad
 L:=\Delta+|A|^2,
 \qquad
 \cQ_R(u,v)
 :=\int_{M\cap B_R}
 \bigl(\langle\nabla u,\nabla v\rangle-|A|^2uv\bigr)\,d\mu,
\]
and
\[
 \cQ_\infty(u,v):=\lim_{R\to\infty}\cQ_R(u,v)
\]
whenever the limit exists.  When both arguments have finite energy, we
write simply
\[
 \cQ(u,v)
 =\int_M
 \bigl(\langle\nabla u,\nabla v\rangle-|A|^2uv\bigr)\,d\mu.
\]
Set
\[
 \|f\|_{\cB}^2
 :=\int_M\bigl(|\nabla f|^2+|A|^2f^2\bigr)\,d\mu,
\]
and let $\cB\subset W^{1,2}_{\mathrm{loc}}(M)$ be the Hilbert-space
completion of $C_c^\infty(M)$ in this norm.

By the spectral result recalled in~\cite[
Proposition~2.5]{Li}, there are $I$ $L^2$-orthonormal bound states
\[
 \varphi_1,\ldots,\varphi_I
 \in W^{1,2}(M)\cap C^\infty_{\mathrm{loc}}(M)\subset\cB
\]
and numbers $\lambda_1,\ldots,\lambda_I<0$ such that
\begin{equation}\label{eq:eigenstates}
 L\varphi_i+\lambda_i\varphi_i=0,
 \qquad
 \|\varphi_i\|_{L^2(M)}=1.
\end{equation}
In particular,
\begin{equation}\label{eq:Q-eigenstates}
 \cQ(\varphi_i,\varphi_j)=\lambda_i\delta_{ij}\leq 0.
\end{equation}
The next lemma is the finite-index replacement for
\cite[Lemma~3.7]{CG}.

\begin{lemma}\label{lem:index-decomposition}
Suppose that $f\in\cB$ satisfies
\begin{equation}\label{eq:Q-orthogonality}
 \cQ(f,\varphi_i)=0,
 \qquad i=1,\ldots,I.
\end{equation}
Then
\[
 \cQ(f,f)\geq 0,
\]
with equality if and only if $Lf=0$.
\end{lemma}

\begin{proof}
 Density of $C_c^\infty(M)$ in $\cB$  implies that $\cQ$
has index $I$ with respect to functions in $\cB$. The span
\[
 \mathcal N:=\Span\{\varphi_1,\ldots,\varphi_I\}
\]
is negative definite by \eqref{eq:Q-eigenstates}.  If
$\cQ(f,f)<0$, then \eqref{eq:Q-orthogonality} implies that $\cQ$
is negative definite on $\mathcal N\oplus\Span\{f\}$. Therefore
$\cQ(f,f)\geq0$.

Assume now that $\cQ(f,f)=0$.  Fix $h\in C_c^\infty(M)$ and define
\begin{equation}\label{eq:h-hat}
 \widehat h
 :=h-\sum_{i=1}^I
 \lambda_i^{-1}\cQ(h,\varphi_i)\varphi_i.
\end{equation}
Using \eqref{eq:Q-eigenstates}, for every $j=1,\ldots,I$ we obtain
\begin{align*}
 \cQ(\widehat h,\varphi_j)
 &=\cQ(h,\varphi_j)
   -\sum_{i=1}^I\lambda_i^{-1}
       \cQ(h,\varphi_i)\cQ(\varphi_i,\varphi_j)\\
 &=\cQ(h,\varphi_j)
   -\lambda_j^{-1}\cQ(h,\varphi_j)\lambda_j
 =0.
\end{align*}
Consequently, $f+t\widehat h$ is $\cQ$-orthogonal to every bound state $\varphi_j$
for all $t\in\R$.  The non-negativity just proved gives
\begin{align*}
 0
 &\leq \cQ(f+t\widehat h,f+t\widehat h)\\
 &=2t\cQ(f,\widehat h)+t^2\cQ(\widehat h,\widehat h),
 \qquad \forall t\in\R.
\end{align*}
The constant term is zero.  A real quadratic polynomial of this form
cannot be nonnegative for every $t$ unless its linear coefficient
vanishes.  Hence $\cQ(f,\widehat h)=0$.  By
\eqref{eq:Q-orthogonality} and \eqref{eq:h-hat},
\[
 \cQ(f,h)=\cQ(f,\widehat h)=0.
\]
Since this holds for every $h\in C_c^\infty(M)$, the definition of
$\cQ$ gives $Lf=0$ weakly.  Elliptic regularity makes $f$ a smooth
Jacobi field.

Conversely, suppose that $Lf=0$ weakly and $f\in\cB$.  Choose
$f_j\in C_c^\infty(M)$ with $f_j\to f$ in $\cB$.  Then
$\cQ(f,f_j)=0$ for every $j$.  Continuity of $\cQ$ on
$\cB\times\cB$ yields
\[
 \cQ(f,f)=\lim_{j\to\infty}\cQ(f,f_j)=0.
\]
This proves the equality characterization.
\end{proof}

\section{Inputs from Chodosh--Gianocca}\label{sec:CG-inputs}

We isolate only the results of \cite[Sections~4--8]{CG} needed below.
Their proofs use the regular end expansions, harmonic-form analysis,
and admissible-pairing identities, but not the index-one hypothesis.
The sole new statement in this section is
\cref{cor:finite-index-7.3}, which replaces orthogonality to one bound
state by orthogonality to all negative bound states.

\paragraph{Harmonic forms and dilation.}
Let $Z=\langle x,\nu\rangle$.  Then $LZ=0$ and
\begin{equation}\label{eq:Z-flux}
 \cQ_\infty(Z,Z)
 =(n-1)\int_M|\nabla x^N|^2\,d\mu>0
\end{equation}
by \cite[Corollary~4.2]{CG}.  As in~\cite[Section~3.1]{CG}, after a
rotation the ends $E_1,\ldots,E_k$ are outer graphs over the fixed
hyperplane
\[
 \Pi:=\{x^N=0\}.
\]
Let $\cH$ be the space in \cite[Definition~5.1]{CG}: its elements are
harmonic $1$-forms for
which some unique $q\in\Pi$ satisfies
\begin{equation}\label{eq:H-asymptotics}
 \nabla^\ell(\omega^\sharp-h_\alpha q^\top)
 =O(r^{1-n-\ell})
 \quad\text{on }E_\alpha.
\end{equation}
Writing $\cH_0$ for the subspace of $\cH$ with $q=0$ and
\begin{equation}\label{eq:E}
 \cE:=\Span\{dx^N\}\subset\cH_0,
\end{equation}
\cite[Lemma~5.2]{CG} gives
\begin{equation}\label{eq:H-dimension}
 \dim\cH_0\geq b_1(M)+k-1,
 \qquad
 \dim\cH\geq b_1(M)+k+n-1.
\end{equation}

\paragraph{Admissible pairings.}
For $\Omega\in\Lambda^{N-4}\R^N$, set
\begin{equation}\label{eq:P-Omega}
 K_\Omega\eta=\star(\Omega\wedge\eta),
 \qquad P_\Omega=\Id+K_\Omega.
\end{equation}
If $P_\Omega\geq0$, let $r=\operatorname{rank}P_\Omega$ and choose $\Theta_1,\ldots,\Theta_r\in\Lambda^2\R^N$ such that
\begin{equation}\label{eq:admissible-factorization}
 P_\Omega=\sum_{a=1}^r\Theta_a\otimes\Theta_a.
\end{equation}
This is called an admissible pairing set, and by \cite[Lemma~6.1]{CG}
\begin{equation}\label{eq:decomposable}
 \langle P_\Omega\eta,\eta\rangle=|\eta|^2
 \quad\text{for every decomposable }\eta.
\end{equation}
\begin{definition}
Let
$\eta\in \Lambda^2\mathbb{R}^N\setminus\{0\}.$
By the canonical form for skew-symmetric operators, there exist pairwise
orthogonal oriented two-planes
$
U_1,\ldots,U_m\subset\mathbb{R}^N
$
and numbers
$
\lambda_1\ge \lambda_2\ge\cdots\ge\lambda_m>0
$
such that
$
\eta=\lambda_1\tau_1+\cdots+\lambda_m\tau_m,
$
where $\tau_i$ denotes the unit volume form of $U_i$.
We say that
$\eta$ is \emph{balanced} if
$
\lambda_1\le \lambda_2+\cdots+\lambda_m.
$
\end{definition}
We also use the following two algebraic
inputs without repeating their proofs.

\begin{proposition}[\cite{CG}, Proposition~6.7]\label{prop:balanced-existence}
Every $N$-dimensional subspace of $\Lambda^2\R^N$, $N\geq4$, contains
a nonzero balanced $2$-form.
\end{proposition}

\begin{proposition}[\cite{CG}, Proposition~6.8]\label{prop:balanced-pairing}
For every nonzero balanced $2$-form $\eta$, there is
$\Omega\in\Lambda^{N-4}\R^N$ such that $P_\Omega\geq0$ and
$P_\Omega\eta=0$.
\end{proposition}

\paragraph{Test functions and nullity.}
For a harmonic $1$-form $\omega$, define
\begin{equation}\label{eq:test-functions}
 u_a^\omega=\langle\nu\wedge\omega^\sharp,\Theta_a\rangle.
\end{equation}
Lemma~7.1 of \cite{CG} gives the pointwise identity
\begin{equation}\label{eq:trace-identity}
 \sum_{a=1}^r
 \bigl(|\nabla u_a^\omega|^2-|A|^2(u_a^\omega)^2\bigr)\,d\mu
 =\frac12\Delta|\omega|^2\,d\mu
  +2(-1)^n d(\Omega_M\wedge\omega\wedge A\omega).
\end{equation}
For $\omega\in\cH$ with asymptotic vector $q$, put
\begin{equation}\label{eq:c-a}
 c_a(q)=\langle e_N\wedge q,\Theta_a\rangle,
 \qquad
 \widetilde u_a^\omega=u_a^\omega-c_a(q)Z.
\end{equation}
Then $\widetilde u_a^\omega\in\cB$ and
\begin{equation}\label{eq:prop-7.2}
 \sum_{a=1}^r
 \cQ(\widetilde u_a^\omega,\widetilde u_a^\omega)
 =-|q|^2(n-1)\int_M|\nabla x^N|^2\,d\mu\leq0
\end{equation}
by \cite[Proposition~7.2]{CG}.

\begin{corollary}\label{cor:finite-index-7.3}
For every $a$ and $i$, the limit
$\cQ_\infty(u_a^\omega,\varphi_i)$ exists.  If
\begin{equation}\label{eq:all-orthogonality}
 \cQ_\infty(u_a^\omega,\varphi_i)=0
 \quad\text{for all }a=1,\ldots,r,
 \quad i=1,\ldots,I,
\end{equation}
then $\omega\in\cH_0$ and $Lu_a^\omega=0$ for every $a$.
\end{corollary}

\begin{proof}
Fix $i\in\{1,\ldots,I\}$. The proof of \cite[Corollary~7.3]{CG} does not rely on the choice of bound states, thus the proof is identical. The required limits exist
and that
\begin{equation}\label{eq:limit-equals-Q}
 \cQ_\infty(u_a^\omega,\varphi_i)
 =\cQ(\widetilde u_a^\omega,\varphi_i).
\end{equation}
Under \eqref{eq:all-orthogonality},
\cref{lem:index-decomposition} makes every summand on the left of
\eqref{eq:prop-7.2} nonnegative.  Equation
\eqref{eq:Z-flux} forces $q=0$, and then equality in
\cref{lem:index-decomposition} gives $Lu_a^\omega=0$ for every $a$.
\end{proof}

Finally, for
\begin{equation}\label{eq:K-Omega}
 \cK_\Omega
 :=\{\omega\in\cH_0:Lu_a^\omega=0
       \text{ for all }a=1,\ldots,r\},
\end{equation}
the nullity characterization \cite[Proposition~8.1]{CG} is
\begin{equation}\label{eq:prop-8.1}
 \cK_\Omega=\Span\{dx^N\}=\cE.
\end{equation}
Its proof contains no spectral or index-one input, so it applies under
the present finite-index hypotheses.

\section{Proof of \cref{thm:main}}\label{sec:main-proof}
Assume that $N=n+1\geq 4$ and that $M^n\subset\R^N$ is a
complete, connected, embedded, non-flat minimal hypersurface with finite total curvature and
$\operatorname{Ind}(M) = I$. Recall that if $M$ is stable, i.e. $I=0$, then it must be flat by  \cite{SZ98}. As such we assume $I\geq 1$.

For $i=1,\ldots,I$, define the linear map
\begin{equation}\label{eq:T-i}
 \mathcal T_i:\cH\longrightarrow\Lambda^2\R^N
\end{equation}
by the Riesz representation theorem
\begin{equation}\label{eq:T-i-definition}
 \langle\mathcal T_i(\omega),\Theta\rangle
 :=\cQ_\infty
 \bigl(\langle\nu\wedge\omega^\sharp,\Theta\rangle,
       \varphi_i\bigr),
 \qquad \Theta\in\Lambda^2\R^N.
\end{equation}
The existence of the limit for every $\Theta$ follows from the first
part of the proof of \cref{cor:finite-index-7.3}: take the admissible
pairing $P_0=\Id$ and factor it using an orthonormal basis of
$\Lambda^2\R^N$.

Set
\begin{equation}\label{eq:D}
 D:=\dim\Lambda^2\R^N=\binom N2
\end{equation}
and retain only the first $I-1$ maps at the initial linear-algebra
stage:
\begin{equation}\label{eq:W}
 W:=\bigcap_{i=1}^{I-1}\ker\mathcal T_i,
\end{equation}
where the intersection is understood to be $\cH$ if $I=1$.

\begin{lemma}\label{lem:W-dimension}

\begin{equation}\label{eq:W-dimension}
 \dim W\geq\dim\cH-(I-1)D.
\end{equation}
\end{lemma}

\begin{proof}
Consider the single linear map
\[
 F:=(\mathcal T_1,\ldots,\mathcal T_{I-1}):
 \cH\longrightarrow
 (\Lambda^2\R^N)^{\oplus(I-1)}.
\]
Then $\ker F=W$, while its target has dimension $(I-1)D$.
The rank--nullity theorem gives
\begin{align*}
 \dim W
 &=\dim\cH-\operatorname{rank}F\\
 &\geq\dim\cH-\dim(\Lambda^2\R^N)^{\oplus(I-1)}\\
 &=\dim\cH-(I-1)D.
\end{align*}
\end{proof}

If $\dim W\geq N+1$,
then there is an $N$-dimensional subspace $V\subset W$ such
that
\begin{equation}\label{eq:V-disjoint-E}
 V\cap\cE=\{0\}.
\end{equation}
Now consider
\[
 \mathcal T_I|_V:V\longrightarrow\Lambda^2\R^N.
\]

\smallskip
\noindent\emph{Case 1: $\mathcal T_I|_V$ is not injective.}
Choose $0\neq\omega\in V$ such that
$\mathcal T_I(\omega)=0$.  Since $V\subset W$, we also have
$\mathcal T_i(\omega)=0$ for $i=1,\ldots,I-1$.  Take $\Omega=0$;
then $P_\Omega=\Id$, and an orthonormal basis
$\Theta_1,\ldots,\Theta_D$ of $\Lambda^2\R^N$ is an admissible pairing
set.  By \eqref{eq:T-i-definition},
\[
 \cQ_\infty(u_a^\omega,\varphi_i)=0
 \quad\text{for every }a=1,\ldots,D
 \text{ and }i=1,\ldots,I.
\]
Corollary~\ref{cor:finite-index-7.3} implies
$\omega\in\cH_0$ and $Lu_a^\omega=0$ for all $a$.  Thus
$\omega\in\cK_0$, and Proposition~8.1, in the form
\eqref{eq:prop-8.1}, gives $\omega\in\cE$.  This contradicts
$0\neq\omega\in V$ and \eqref{eq:V-disjoint-E}.

\smallskip
\noindent\emph{Case 2: $\mathcal T_I|_V$ is injective.}
Then
\[
 \dim\mathcal T_I(V)=\dim V=N.
\]
Proposition~\ref{prop:balanced-existence} gives
$0\neq\omega\in V$ such that
$ \eta:=\mathcal T_I(\omega)
$
is a nonzero balanced $2$-form. Proposition
\ref{prop:balanced-pairing} gives
$\Omega\in\Lambda^{N-4}\R^N$ for which $P_\Omega$ is positive
semidefinite and
\begin{equation}\label{eq:P-kills-eta}
 P_\Omega\eta=0.
\end{equation}
Choose an admissible pairing set
\[
 P_\Omega=\sum_{a=1}^r\Theta_a\otimes\Theta_a.
\]
Because $\omega\in V\subset W$, for $i=1,\ldots,I-1$ and every $a$,
\[
 \cQ_\infty(u_a^\omega,\varphi_i)
 =\langle\mathcal T_i(\omega),\Theta_a\rangle=0.
\]
For $i=I$, \eqref{eq:P-kills-eta} gives
\begin{align*}
 0
 &=\langle P_\Omega\eta,\eta\rangle\\
 &=\sum_{a=1}^r\langle\eta,\Theta_a\rangle^2\\
 &=\sum_{a=1}^r
   \cQ_\infty(u_a^\omega,\varphi_I)^2.
\end{align*}
Hence
$\cQ_\infty(u_a^\omega,\varphi_I)=0$ for every $a$.  Thus the
hypothesis of Corollary~\ref{cor:finite-index-7.3} holds for all $I$
bound states.  As in Case~1, that corollary and
\eqref{eq:prop-8.1} give
$\omega\in\cK_\Omega=\cE$, contradicting
$0\neq\omega\in V$ and \eqref{eq:V-disjoint-E}.

Both cases are impossible, therefore
\begin{equation}\label{eq:H-upper}
 N\geq \dim W\geq\dim\cH-(I-1)D.
\end{equation}
Combining \eqref{eq:H-upper} with the unchanged lower bound
\eqref{eq:H-dimension}, and using $N=n+1$, gives
\[
 b_1(M)+k+n-1
 \leq (I-1)\binom N2+n+1.
\]
Thus
\[
 \Ind(M)\geq
 1+\frac{b_1(M)+k-2}{D}.
\]

\begin{remark}
    When $I=1$, i.e. the intersection $W$ reduced to $\cH$ the proof is identical to that of \cite[Theorem~1.1]{CG}. This also explains why our estimate is ``sharp'' for the catenoid.
\end{remark}

\section{$\operatorname{Spin}(7)$ estimates in $\R^8$}\label{sec:spin7}

We prove \cref{thm:spin7-index-bound,thm:spin7-index-one}.  The common
algebraic ingredient is the rank-seven summand in the
$\operatorname{Spin}(7)$ decomposition of two-forms.  We record the
normalization because both the number of test functions and the final
dimension counts depend on it.

\subsection{The $\operatorname{Spin}(7)$ decomposition of two-forms}

Fix the standard oriented orthonormal coframe
$e^1,\ldots,e^8$ on $\R^8$, and write
$e^{i_1\cdots i_j}=e^{i_1}\wedge\cdots\wedge e^{i_j}$.  Consider the
Cayley form
\begin{align}\label{eq:cayley-form}
 \Phi={}&e^{1234}+e^{1256}+e^{1278}+e^{1357}
          -e^{1368}-e^{1458}-e^{1467}\notag\\
       &{}-e^{2358}-e^{2367}-e^{2457}+e^{2468}
          +e^{3456}+e^{3478}+e^{5678}.
\end{align}
The self-adjoint operator
\[
 T_\Phi:\Lambda^2\R^8\longrightarrow\Lambda^2\R^8,
 \qquad T_\Phi(\eta)=\star(\Phi\wedge\eta),
\]
has eigenvalues $3$ and $-1$, of multiplicities $7$ and $21$,
respectively; see, for example, \cite[Section~2.1]{Wal}.  Thus
\begin{equation}\label{eq:spin7-splitting}
 \Lambda^2\R^8=\Lambda^2_7\oplus\Lambda^2_{21},
 \qquad
 \Lambda^2_7=\ker(T_\Phi-3\Id),
 \qquad
 \Lambda^2_{21}=\ker(T_\Phi+\Id).
\end{equation}
If $P_7$ denotes orthogonal projection onto $\Lambda^2_7$, then the
operator of \eqref{eq:P-Omega} associated with $\Omega=\Phi$ is
\begin{equation}\label{eq:P-Phi}
 P_\Phi=\Id+T_\Phi=4P_7\geq0.
\end{equation}

For completeness, an admissible pairing set for $P_\Phi$ is
\begin{align*}
 \Theta_1&=e^{12}+e^{34}+e^{56}+e^{78},
 &\Theta_2&=e^{13}-e^{24}+e^{57}-e^{68},\\
 \Theta_3&=e^{14}+e^{23}-e^{58}-e^{67},
 &\Theta_4&=e^{15}-e^{26}-e^{37}+e^{48},\\
 \Theta_5&=e^{16}+e^{25}+e^{38}+e^{47},
 &\Theta_6&=e^{17}-e^{28}+e^{35}-e^{46},\\
 \Theta_7&=e^{18}+e^{27}-e^{36}-e^{45}.
\end{align*}
Indeed, the forms $\frac12\Theta_a$ form an orthonormal basis of
$\Lambda^2_7$, and hence
\begin{equation}\label{eq:P-Phi-factorization}
 P_\Phi=4P_7=\sum_{a=1}^7\Theta_a\otimes\Theta_a.
\end{equation}

If $\eta$ is decomposable, then $\eta\wedge\eta=0$, so
\eqref{eq:decomposable} and \eqref{eq:P-Phi} give
\begin{equation}\label{eq:P7-decomposable}
 4|P_7\eta|^2=|\eta|^2.
\end{equation}
Consequently, for every unit vector $\nu\in S^7$, the map
\begin{equation}\label{eq:R-nu-spin7}
 R_\nu:\nu^\perp\longrightarrow\Lambda^2_7,
 \qquad R_\nu(q)=P_7(\nu\wedge q),
\end{equation}
satisfies $|R_\nu(q)|=\frac12|q|$.  It is therefore injective, and it
is an isomorphism because both spaces have dimension seven.

\subsection{The improved index estimate}

\begin{proof}[Proof of \cref{thm:spin7-index-bound}]
The assertion is immediate if $\Ind(M)=\infty$, so write
$I=\Ind(M)<\infty$ and let $\varphi_1,\ldots,\varphi_I$ be the negative
bound states.  Use the admissible pairing set
$\Theta_1,\ldots,\Theta_7$ in \eqref{eq:P-Phi-factorization} and define
\[
 \mathcal F:\cH\longrightarrow\R^{7I},
 \qquad
 \mathcal F(\omega)
 =\bigl(\cQ_\infty(u_a^\omega,\varphi_i)
      \bigr)_{1\leq a\leq7,\,1\leq i\leq I}.
\]
If $\dim\cH>7I+1$, then $\dim\ker\mathcal F>1$.  Since
$\cE=\Span\{dx^8\}$ is one-dimensional, there is a nonzero
$\omega\in\ker\mathcal F\setminus\cE$.  On the other hand,
\cref{cor:finite-index-7.3} gives $\omega\in\cH_0$ and
$Lu_a^\omega=0$ for all $a$, while \eqref{eq:prop-8.1} gives
$\omega\in\cK_\Phi=\cE$, a contradiction.  Hence
\[
 \dim\cH\leq7I+1.
\]
Because $n=7$, \eqref{eq:H-dimension} yields
\[
 b_1(M)+k+6\leq\dim\cH\leq7I+1,
\]
which is equivalent to \eqref{eq:spin7-index-bound}. 
\end{proof}

\subsection{Proof of \cref{thm:spin7-index-one}}

\begin{proof}
We follow the proof of \cite[Theorem~10.1]{CG} in constructing the
harmonic $1$-forms and the associated test functions. We spell out the
dimension-dependent details.

Finite total curvature gives regularity at infinity; see
\cite[Section~2]{Sch83}.  Let $E_1,\ldots,E_k$ be the resulting regular
ends, let $\nu_\alpha\in S^7$ be the limiting unit normal of
$E_\alpha$, and put
\[
 V:=\Span\{\nu_\alpha:1\leq\alpha\leq k\},
 \qquad d:=\dim V.
\]
If $d=1$, all ends are parallel and the parallel-end version of the
argument proving \cite[Theorem~1.1]{CG} applies.  We may therefore
assume $d\geq2$.
For a linear map $S:V\to\Lambda^2_7$, define
\begin{equation}\label{eq:q-alpha-spin7}
q_\alpha:=R_{\nu_\alpha}^{-1}(S\nu_\alpha)
\in\nu_\alpha^\perp,
\end{equation}
where $R_{\nu_\alpha}$ is the linear isomorphism defined in
\textup{(41)}. The harmonic-form construction used in
\cite[Lemma~5.2 and Theorem~10.1]{CG} shows that, for every such
$S$, there exists a harmonic $1$-form $\omega$ satisfying
\begin{equation}\label{eq:omega-spin7-asymptotics}
\left|\nabla^\ell(\omega^\sharp-q_\alpha^\top)\right|
=O(r^{-6-\ell})
\quad\text{on }E_\alpha.
\end{equation}
Let $\widehat{\cH}$ be the space of all harmonic $1$-forms satisfying
\eqref{eq:omega-spin7-asymptotics} for some
$S\in\operatorname{Hom}(V,\Lambda^2_7)$. Since the vectors
$\nu_\alpha$ span $V$ and each $R_{\nu_\alpha}$ is injective, the
parameter $S$ is uniquely determined by $\omega$. Thus the assignment
$\omega\mapsto S$ is linear and surjective, and its kernel contains
the space of $L^2$ harmonic $1$-forms by the same argument as in \cite[Lemma 5.2]{CG}, whose dimension is at least
$b_1(M)+k-1$. Therefore
\begin{equation}\label{eq:Hhat-spin7-lower}
\dim\widehat{\cH}
\geq b_1(M)+k-1
+\dim\operatorname{Hom}(V,\Lambda^2_7)
=b_1(M)+k-1+7d.
\end{equation}

Using the seven forms in \eqref{eq:P-Phi-factorization}, set
\[
 u_a^\omega=\langle\nu\wedge\omega^\sharp,\Theta_a\rangle,
 \qquad a=1,\ldots,7.
\]
On $E_\alpha$ the limiting value is
\begin{align*}
\langle\nu_\alpha\wedge q_\alpha^T,\Theta_a\rangle
&=\langle\nu_\alpha\wedge q_\alpha,\Theta_a\rangle\\
 &=\langle P_7(\nu_\alpha\wedge q_\alpha),\Theta_a\rangle\\
 &=\langle S\nu_\alpha,\Theta_a\rangle\\
 &=\langle\nu_\alpha,S^*\Theta_a\rangle.
\end{align*}

It agrees with the limiting value of the translation Jacobi field
\[
 t_a:=\langle\nu,S^*\Theta_a\rangle.
\]
Therefore
\begin{equation}\label{eq:u-tilde-spin7}
 \widetilde u_a^\omega:=u_a^\omega-t_a\in\cB.
\end{equation}

We verify here the zero-flux assertion for translations.  For a fixed
$p\in\R^8$, put $t_p=\langle\nu,p\rangle$.  Since ambient
translations preserve minimality,
\[
 Lt_p=0,
 \qquad
 \nabla_Xt_p=\langle D_X\nu,p\rangle
             =-\langle AX,p^\top\rangle.
\]
Regularity of the ends gives $|A|=O(r^{-7})$, while
$|M\cap\partial B_R|=O(R^6)$ and $t_p=O(1)$.  Integration by parts,
for regular values of $R$, therefore yields
\begin{align}\label{eq:translation-zero-flux}
 \cQ_R(t_p,t_q)
 &=\int_{M\cap\partial B_R}t_p\nabla_\vartheta t_q
   =O(R^6R^{-7})=O(R^{-1})\longrightarrow0.
\end{align}
The same calculation gives
$\cQ_R(\widetilde u_a^\omega,t_a)\to0$, since
$\widetilde u_a^\omega=O(r^{-6})$ along every end.  Finally, the
boundary terms obtained by integrating the trace identity
\eqref{eq:trace-identity} are $O(R^{-1})$: in the first one this uses
$\nabla\omega=O(r^{-7})$, and in the $\Phi_M\wedge\omega\wedge
A\omega$ term it uses $A=O(r^{-7})$.  Thus the proof of
\cite[Proposition~7.2]{CG}, with the dilation field replaced by the
translation fields $t_a$, gives
\begin{equation}\label{eq:spin7-Q-zero}
 \sum_{a=1}^7
 \cQ(\widetilde u_a^\omega,\widetilde u_a^\omega)=0.
\end{equation}

Let $\varphi$ be the unique negative bound state.  Suppose that
\begin{equation}\label{eq:spin7-index-orthogonality}
 \cQ(\widetilde u_a^\omega,\varphi)=0,
 \qquad a=1,\ldots,7.
\end{equation}
The index-one decomposition and \eqref{eq:spin7-Q-zero} imply
$L\widetilde u_a^\omega=0$ for every $a$.  Since $Lt_a=0$, we also
have $Lu_a^\omega=0$.

Fix an end $E_\alpha$ and define
\[
 \widehat\omega
 :=\omega-d\langle q_\alpha,x\rangle.
\]
By \eqref{eq:omega-spin7-asymptotics}, $\widehat\omega$ decays to $0$ on
$E_\alpha$.  Moreover, each
$u_a^{d\langle q_\alpha,x\rangle}$ is a translation Jacobi field, so
$Lu_a^{\widehat\omega}=0$. Note that the argument of
\cite[Section~8]{CG} does not require $\omega$ to decay on every end but only one end, thus the same argument now gives
\[
\widehat\omega=\lambda d\langle\nu_\alpha,x\rangle.
\]
Consequently
\begin{equation}\label{eq:spin7-coordinate-kernel}
 \omega=d\langle q_\alpha+\lambda\nu_\alpha,x\rangle
 \in\Span\{dx^1,\ldots,dx^8\}.
\end{equation}

It remains to make the dimension count explicit.  Define a linear map
\[
 F:\widehat\cH\longrightarrow\R^7,
 \qquad
 F(\omega)
 =\bigl(\cQ(\widetilde u_1^\omega,\varphi),\ldots,
         \cQ(\widetilde u_7^\omega,\varphi)\bigr).
\]
There are seven components because $\dim\Lambda^2_7=7$.  By
\eqref{eq:spin7-coordinate-kernel},
\[
 \ker F\subset\Span\{dx^1,\ldots,dx^8\},
\]
Therefore
\begin{equation}\label{eq:Hhat-spin7-upper}
 \dim\widehat\cH
 =\dim\ker F+\operatorname{rank}F
 \leq\dim\Span\{dx^1,\ldots,dx^8\}+7
 \leq8+7=15.
\end{equation}
Combining \eqref{eq:Hhat-spin7-lower} and
\eqref{eq:Hhat-spin7-upper} gives
\[
 b_1(M)+k-1+7d\leq15.
\]
Since $d\geq2$, this implies
\[
 b_1(M)+k\leq2,
 \qquad\text{and hence}\qquad k\leq2.
\]
Also, $2\leq d\leq k$.
By Schoen's two-end theorem~\cite{Sch83}, 
$M$ is a higher-dimensional catenoid.
\end{proof}

\begin{remark}
The coefficient obtainable from a fixed admissible-pairing argument is controlled by the rank of the positive semidefinite self-adjoint operator $P$ on $\Lambda^2\R^N$ satisfying
\[
 \langle P\eta,\eta\rangle=|\eta|^2
 \qquad\text{for every decomposable }\eta\in\Lambda^2\R^N.
\]
Indeed, write $r=\operatorname{rank}P$ and choose a factorization $P=\sum_{a=1}^r\Theta_a\otimes\Theta_a$.  The number $r$ is precisely the number of scalar test functions arising from this pairing.  Define
\[
 F:TS^{N-1}\longrightarrow S^{N-1}\times\R^r,
 \qquad F(x,v)=\bigl(x,(\langle x\wedge v,\Theta_a\rangle)_{a=1}^r\bigr).
\]
This is a smooth bundle map covering the identity on $S^{N-1}$, and its restriction $F_x:T_xS^{N-1}=x^\perp\to\R^r$ is linear.  For $|x|=1$ and $v\perp x$, the decomposable identity gives
\[
 |F_x(v)|^2=\sum_{a=1}^r\langle x\wedge v,\Theta_a\rangle^2
 =\langle P(x\wedge v),x\wedge v\rangle
 =|x\wedge v|^2=|v|^2.
\]
Thus every $F_x$ is injective, and consequently $r\geq N-1$.  If $r=N-1$, then every $F_x$ is a linear isomorphism.  
In particular, $S^{N-1}$ is parallelizable.  Since the only parallelizable spheres are $S^0,S^1,S^3$, and $S^7$ \cite{BM58}, equality $r=N-1$ can occur in the range $N\geq4$ only when $N=4$ or $N=8$. Thus outside these two dimensions one should not expect optimization of $P$ alone either to attain the extremal coefficient $1/(N-1)$ or to extend this fixed-pairing argument to classify immersed index-one hypersurfaces as catenoids.
\end{remark}

\end{document}